\documentclass[a4paper,11pt]{amsart}

\usepackage{amsmath,amssymb,amsthm}
\usepackage{indentfirst}
\usepackage{graphicx}
\usepackage{enumerate}
\usepackage[colorlinks, citecolor=green, linkcolor=red]{hyperref}
\usepackage{cite}
\hypersetup{pdfborder={0 0 0}} 

\newtheorem{theorem}{Theorem}[section]
\newtheorem{corollary}[theorem]{Corollary}

\newtheorem{proposition}[theorem]{Proposition}

\newtheorem{remark}[theorem]{Remark}

\theoremstyle{definition}

\newcommand{\R}{\mathbb{R}}
\newcommand{\Sp}{\mathbb{S}}
\newcommand{\Hi}{\mathbb{H}}
\newcommand{\B}{\mathbb{B}}
\DeclareMathOperator{\oRic}{Ric}

\DeclareMathOperator{\tr}{tr}
\newcommand{\oi}{\mathbf{i}}

\title[Curvature Inhomogeneous Submanifolds with Constant Ricci Eigenvalues]{Examples of Curvature Inhomogeneous Submanifolds with Constant Ricci Eigenvalues}

\author[J. Q. Ge]{Jianquan Ge}
\address{School of Mathematical Sciences, Beijing Normal University, Beijing 100875, P. R. China}
\email{jqge@bnu.edu.cn}
\author[Y. Y. Zhao]{Yuyang Zhao$^{*}$}
\address{School of Mathematical Sciences, Beijing Normal University, Beijing 100875, P. R. China}
\email{yyzhao24@mail.bnu.edu.cn}

\thanks{$^{*}$ Corresponding author.}
\date{}
\thanks{J. Q. Ge is partially supported by the NSFC (No. 12571049) and the Fundamental Research Funds for the Central Universities.}

\begin{document}

\begin{abstract}
	We construct two families of curvature inhomogeneous Riemannian manifolds with constant Ricci eigenvalues. The first, derived from the Einstein warped products, has two distinct Ricci eigenvalues and admits a local isometric immersion of minimum codimension two. The second, arising from the Riemannian Schwarzschild--Tangherlini manifold, has $k+1$ distinct Ricci eigenvalues and admits an isometric embedding of codimension $k+2$, which is the smallest within the adapted product class.
\end{abstract}

\maketitle

\section{Introduction}\label{sec:intro}
Following Singer \cite{Singer1960}, a Riemannian manifold $(M^{n},g)$ is \emph{curvature homogeneous} if for any two points $p,q\in M$, there exists a linear isometry $J\colon T_pM\to T_qM$ such that $J^{*}R_q=R_p$, where $R$ denotes the Riemann curvature tensor. Curvature homogeneity implies that $(M^{n},g)$ has \emph{constant Ricci eigenvalues}, i.e., the eigenvalues of the Ricci endomorphism $\widehat{\oRic}$ are constant functions on $M$. For $n=3$, the converse holds because $R$ is determined algebraically by $\widehat{\oRic}$. For $n\geq 4$, however, our examples below show that constancy of the Ricci eigenvalues is strictly weaker than curvature homogeneity. We call a manifold \emph{curvature inhomogeneous} if it is not curvature homogeneous.

For hypersurfaces of real space forms, Ge and Zhao \cite{GZ-CREH} showed that constant Ricci eigenvalues are equivalent to curvature homogeneity. Tsukada's classification of curvature homogeneous hypersurfaces \cite{Tsukada88}, together with the remaining rank-two cases in $\Sp^4$ and $\Hi^4$, recently settled by Bryant--Florit--Ziller \cite{BFZ}, therefore gives a complete classification of hypersurfaces with constant Ricci eigenvalues.

This equivalence immediately yields an obstruction to isometric immersions into space forms of codimension one.

\begin{corollary}[{\cite[Corollary 1.2]{GZ-CREH}}]\label{cor:nohyp}
	A connected curvature inhomogeneous Riemannian manifold of dimension at least four with constant Ricci eigenvalues cannot be isometrically immersed as a hypersurface into any real space form.
\end{corollary}

As curvature homogeneity is a local property, Corollary~\ref{cor:nohyp} applies to any connected open subset that is curvature inhomogeneous and has constant Ricci eigenvalues. It remains natural to ask whether this obstruction disappears in higher codimension and, if so, what the minimum codimension is. We answer this by constructing curvature inhomogeneous Riemannian manifolds with constant Ricci eigenvalues that admit isometric immersions into Euclidean space of small, explicitly controlled codimension.

The Einstein warped products used below originate from Dajczer--Onti--Vlachos \cite{DOV}, and their complete Ricci-flat case is exactly the Riemannian Schwarzschild--Tangherlini manifolds, named after Schwarzschild \cite{Schwarzschild16} and Tangherlini \cite{Tangherlini63}. The new contributions here are the verification of their curvature inhomogeneity, the proper embedding result for the Riemannian Schwarzschild--Tangherlini manifolds, the cylindrical and spherical product constructions with at least two distinct constant Ricci eigenvalues, and the sharp codimension bounds for immersions adapted to the product structure.

Our first result produces examples with two distinct constant Ricci eigenvalues and minimum local Euclidean codimension two.

\begin{theorem}\label{thm:1}
	Let $n\geq 4$ and $\ell\geq 1$. There exists a Riemannian product
	$$M^{n+\ell}=N^n\times\R^\ell,$$
	where $(N^n,g_N)$ is an Einstein warped product satisfying $\oRic^N=\rho\,g_N$ for some $\rho\neq 0$ (see Subsection \ref{sec:EWP-def}), with the following properties:
	\begin{enumerate}[\rm(i)]
		\item $M^{n+\ell}$ has exactly two distinct constant Ricci eigenvalues, namely $0$ and $\rho$, with multiplicities $\ell$ and $n$;
		\item $M^{n+\ell}$ is curvature inhomogeneous;
		\item $M^{n+\ell}$ admits a local isometric immersion into $\R^{n+\ell+2}$ with flat normal bundle, and the codimension two is the smallest.
	\end{enumerate}
\end{theorem}

Our second result produces examples with arbitrarily many distinct constant Ricci eigenvalues and an explicit extrinsic product embedding.

\begin{theorem}\label{thm:2}
	Let $n\geq 4$ and $k\geq 1$. Let $n_j\geq 2$ and $r_j>0$ be chosen so that the constants $\rho_j:=\frac{n_j-1}{r_j^2}$, $1\leq j\leq k$, are pairwise distinct. There exists a complete Riemannian product
	$$W_k^{m}=N_0^n\times\prod_{j=1}^k\Sp^{n_j}(r_j),\qquad m=n+\sum_{j=1}^k n_j,$$
	where $(N_0^n,g_{N_0})$ is the complete Ricci-flat Riemannian Schwarzschild--Tangherlini manifold (see Subsection~\ref{sec:RFS}), with the following properties:
	\begin{enumerate}[\rm(i)]
		\item $W_k^{m}$ has exactly $k+1$ distinct constant Ricci eigenvalues, namely $0,\rho_1,\ldots,\rho_k$, with multiplicities $n,n_1,\ldots,n_k$;
		\item $W_k^{m}$ is curvature inhomogeneous;
		\item $W_k^{m}$ admits a global isometric embedding into $\R^{m+k+2}$ with flat normal bundle.
	\end{enumerate}
\end{theorem}

The codimension in Theorem~\ref{thm:2}(iii) is sharp within a natural class, and the constraint on lower-codimension immersions can be made precise. Write the factors of $W_k^m$ as $P_0=N_0^n$ and $P_j=\Sp^{n_j}(r_j)$ for $1\leq j\leq k$, so that $TW_k^m=\bigoplus_{j=0}^k TP_j$. The second fundamental form $\alpha$ of an isometric immersion of $W_k^m$ is \emph{adapted} to the product structure if its mixed terms vanish, i.e., $\alpha(TP_i,TP_j)=0$ for $i\neq j$.

\begin{theorem}\label{thm:3}
Let $W_k^m=N_0^n\times\prod_{j=1}^k\Sp^{n_j}(r_j)$ as above, with $n\geq 4$ and $n_j\geq 2$.
\begin{enumerate}[\rm(a)]
\item Every local isometric immersion of $W_k^m$ into Euclidean space with adapted second fundamental form has codimension at least $k+2$. This holds regardless of whether the $\rho_j$ are distinct.
\item If the $\rho_j$ are pairwise distinct and $\Phi\colon W_k^m\to\R^{m+q}$ is a local isometric immersion with $q\leq k+1$, then its second fundamental form is non-adapted, and its normal bundle is nonflat.
\end{enumerate}
\end{theorem}

The paper is organized as follows. Section~\ref{sec:EWP} recalls the Einstein warped products of Dajczer--Onti--Vlachos~\cite{DOV} and the Riemannian Schwarzschild--Tangherlini case, together with their codimension-two isometric immersions and curvature computations. Section~\ref{sec:proofs} proves Theorems~\ref{thm:1}--\ref{thm:3}. Further questions are collected in Section~\ref{sec:questions}.

\section{Preliminaries on Einstein warped product submanifolds}\label{sec:EWP}
In this section, we recall the warped product models of Dajczer--Onti--Vlachos \cite{DOV}, including their local isometric immersions of codimension two. The new ingredients are the curvature inhomogeneous criterion in Proposition~\ref{pro:N} and the global embedding result for Riemannian Schwarzschild--Tangherlini manifolds in Proposition~\ref{pro:embedding}.

\subsection{The Einstein warped products of codimension two}\label{sec:EWP-def}
Fix $n\geq4$ and $\rho\in\R$. By \cite[Example 1(a)]{DOV}, there exist a constant $c\neq0$ and a nonconstant smooth function $\varphi\colon I\to(0,\infty)$ on an open interval $I$ containing $0$ such that
\begin{equation}\label{eq:varphi'}
	\varphi'(t)^{2}=1-\frac{\rho}{n-1}\,\varphi(t)^{2}+\frac{c}{\varphi(t)^{\,n-3}},
\end{equation}
and
\begin{equation}\label{eq:size}
	0<\varphi'(t)^2<1,\qquad t\in I.
\end{equation}
In what follows, we allow any choice of $c$ and $\varphi$ with these
properties.

Let $L^2$ be a coordinate domain with coordinates $(t,u)$, where
$t\in I$, endowed with the metric $g_L=dt^2+\varphi'(t)^2\,du^2$. Define the warped product \cite{DOV}
$$N^n=L^2\times_{\varphi}\Sp^{n-2},\qquad g_N=dt^2+\varphi'(t)^2\,du^2+\varphi(t)^2 g_{\Sp^{n-2}},$$
where $\Sp^{n-2}$ carries the unit round metric and $\varphi(t)$ is the warping function. The assumption $c\neq 0$ excludes the space form case, since \eqref{eq:K_N} below shows that $c=0$ would force every sectional curvature of $N^n$ to equal $\frac{\rho}{n-1}$.

According to \cite[Proposition~4]{DOV}, the warped product $(N^n,g_N)$ is Einstein, namely $\oRic^N=\rho\,g_N$. It admits a local isometric immersion into Euclidean space as an \emph{$(n-2)$-rotational submanifold} with flat normal bundle, as shown in \cite[Example 1(a)]{DOV}. We recall this construction in the form used below.

Split $\R^{n+2}=\R^{3}\times\R^{n-1}$ and let $\mathrm{SO}(n-1)$ act by rotations on the second factor while fixing the axis $\R^{3}$ pointwise. Fixing a unit vector $e\in\R^{n-1}$, a \emph{profile} is an immersed surface
$$\gamma=(h,\varphi)\colon L^{2}\longrightarrow\R^{4}=\R^{3}\oplus\mathrm{span}\{e\},
\qquad h\colon L^{2}\to\R^{3},\quad \varphi=\langle\gamma,e\rangle>0,$$
where the last condition means that $\gamma(L^{2})$ is disjoint from the axis. The $(n-2)$-rotational submanifold generated by $\gamma$ is the orbit of $\gamma(L^{2})$ under $\mathrm{SO}(n-1)$, parametrized by
\begin{equation}\label{eq:rot-param}
	f\colon L^{2}\times\Sp^{n-2}\longrightarrow\R^{n+2},\qquad
	f(x,y)=\bigl(h(x),\,\varphi(x)\,\phi(y)\bigr),
\end{equation}
where $\phi\colon\Sp^{n-2}\hookrightarrow\R^{n-1}$ is the inclusion of the unit round sphere, with induced metric $f^{*}\langle\cdot,\cdot\rangle=\gamma^{*}\langle\cdot,\cdot\rangle+\varphi^{2}g_{\Sp^{n-2}}$. Hence, $f$ induces $g_N$ precisely when the profile is isometric to the base, $\gamma^{*}\langle\cdot,\cdot\rangle=g_L$. Since $d\varphi=\varphi'\,dt$, this amounts to
\begin{equation}\label{eq:h-metric}
	h^{*}\langle\cdot,\cdot\rangle=(1-\varphi'^{2})\,dt^{2}+\varphi'^{2}\,du^{2}.
\end{equation}
The right-hand side of \eqref{eq:h-metric} is positive definite precisely when $0<\varphi'^2<1$. Moreover, \eqref{eq:varphi'} and \eqref{eq:size} imply that the metric in \eqref{eq:h-metric} is real analytic, and hence it admits a local isometric immersion $h$ into $\R^3$ by the Cartan--Janet theorem \cite{HanHong}.

To detect curvature inhomogeneity in this and the subsequent models, we use the \emph{Kretschmann scalar}, namely the squared norm of the Riemann curvature tensor.

\begin{proposition}\label{pro:N}
	The Kretschmann scalar of the Einstein warped product $(N^n,g_N)$ is
	\begin{equation}\label{eq:RN}
		|R^N|^2=\frac{2n}{n-1}\rho^2+(n-1)(n-2)^2(n-3)\frac{c^2}{\varphi^{2n-2}}.
	\end{equation}
	In particular, $(N^n,g_N)$ is curvature inhomogeneous.
\end{proposition}

\begin{proof}
	Take the orthonormal frame
	$$e_1=\partial_t,\qquad e_2=\frac{1}{\varphi'}\partial_u,\qquad e_\alpha=\frac{1}{\varphi}\bar e_\alpha,\quad 3\leq\alpha\leq n,$$
	with $\{\bar e_\alpha\}$ a $g_{\Sp^{n-2}}$-orthonormal frame on the fibre, so that $e_1,e_2$ are tangent to the base $L^2$ and $e_3,\ldots,e_n$ to the fibre $\Sp^{n-2}$.

	For the warped product $N^n=L^2\times_{\varphi}\Sp^{n-2}$, the curvature formulas of \cite[Proposition 3.2]{Chen2017} give three types of sectional curvatures
	\begin{equation}\label{eq:K}
		K_{12}=K^{L^2},\quad K_{i\alpha}=-\frac{H^{\varphi}(e_i,e_i)}{\varphi}\ (i=1,2),\quad K_{\alpha\beta}=\frac{K^{\Sp^{n-2}}-|\nabla\varphi|^2}{\varphi^2},
	\end{equation}
	where $K^{L^2}$ and $K^{\Sp^{n-2}}$ are the sectional curvatures of $L^2$ and $\Sp^{n-2}$, and $H^{\varphi}$ is the Hessian of $\varphi$ on $(L^2,g_L)$.

	Since $\varphi$ depends only on $t$, on $(L^2,g_L)$ one computes
	$$\nabla\varphi=\varphi'\,\partial_t,\qquad |\nabla\varphi|^2=\varphi'^2,\qquad H^{\varphi}=\varphi''\,g_L,\qquad K^{L^2}=-\frac{\varphi'''}{\varphi'},$$
	while $K^{\Sp^{n-2}}=1$. Substituting these identities into \eqref{eq:K} and using \eqref{eq:varphi'}, we obtain
	\begin{equation}\label{eq:K_N}
		\begin{aligned}
			K_{12}&=-\frac{\varphi'''}{\varphi'}=\frac{\rho}{n-1}-\frac{(n-2)(n-3)c}{2\varphi^{n-1}},\\
			K_{1\alpha}=K_{2\alpha}&=-\frac{\varphi''}{\varphi}=\frac{\rho}{n-1}+\frac{(n-3)c}{2\varphi^{n-1}},\\
			K_{\alpha\beta}&=\frac{1-\varphi'^2}{\varphi^2}=\frac{\rho}{n-1}-\frac{c}{\varphi^{n-1}}.
		\end{aligned}
	\end{equation}

	By \cite[Proposition 3.2]{Chen2017}, the curvature operator is diagonal in this frame. Hence,
	$$|R^N|^2=4\left(K_{12}^2+2(n-2)K_{1\alpha}^2+\frac{(n-2)(n-3)}{2}K_{\alpha\beta}^2\right),$$
	and substitution of \eqref{eq:K_N} gives \eqref{eq:RN}. Since $c\neq0$ and $\varphi'\neq0$, $|R^N|^2$ is nonconstant on every nonempty open set, whereas curvature homogeneity would make every scalar curvature contraction constant. Hence, $(N^n,g_N)$ is nowhere locally curvature homogeneous and, in particular, curvature inhomogeneous.
\end{proof}

By Corollary~\ref{cor:nohyp}, or alternatively by Ryan's classification of Einstein hypersurfaces~\cite{Ryan1969}, $(N^n,g_N)$ admits no local isometric immersion as a hypersurface in any real space form. Codimension zero is also impossible, since a codimension-zero Euclidean isometric immersion would be a local isometry and would force $N^n$ to be flat. Consequently, the local isometric immersion $f$ given in \eqref{eq:rot-param} has minimum codimension two.

\subsection{The complete Ricci-flat case: Riemannian Schwarzschild--Tangherlini manifolds}\label{sec:RFS}
Fix $n\geq4$ and set $\rho=0,~\mu=\frac{n-3}{2}$, and $c=-\mu^{n-3}$. Following the complete construction of Dajczer--Onti--Vlachos~\cite[Example~2]{DOV}, let $\varphi\in C^\infty([0,\infty))$ be the unique nonconstant solution of
\begin{equation}\label{eq:varphi-RF}
	\varphi'^{2}=1-\left(\frac{\mu}{\varphi}\right)^{n-3},\qquad \varphi(0)=\mu,\qquad \varphi'(0)=0,
\end{equation}
characterized by $\varphi'(t)>0$ for every $t>0$. Differentiating \eqref{eq:varphi-RF} shows that $\varphi$ is the solution of the initial value problem
\begin{equation}\label{eq:varphi-RF2}
	\varphi''=\frac{n-3}{2}\,\frac{\mu^{n-3}}{\varphi^{\,n-2}},\qquad \varphi(0)=\mu,\qquad \varphi'(0)=0.
\end{equation}
Since \eqref{eq:varphi-RF2} is invariant under $t\mapsto-t$, uniqueness shows that $\varphi$ is even. Moreover, $\varphi''(0)=1$ and $\varphi''>0$, so $\varphi'$ is odd and $\varphi'>0$ on $(0,\infty)$. Consequently, $\varphi\geq\mu$ and $|\varphi'|<1$, so the solution neither reaches the singular set $\varphi=0$ nor escapes to infinity in finite time. It therefore extends smoothly to all of $\R$. In particular,
\begin{equation}\label{eq:varphi-expansion}
	\varphi(t)=\mu+\frac12t^2+O(t^4)\qquad\text{as }t\to0.
\end{equation}
The even extension is used to verify smoothness at the origin $o$. In what follows, $t\geq0$ denotes the radial coordinate in the polar description of $L_0^2\cong\R^2$ below.

Endow $L_0^2\setminus\{o\}$ with the rotationally invariant metric
$$g_{L_0}=dt^2+\varphi'(t)^2\,d\theta^2,\qquad \theta\in [0,2\pi).$$
Since $\varphi'$ is odd with $\varphi''(0)=1$, this metric extends smoothly across $o$ by \cite[Section~1.4]{Petersen}, and $L_0^2$ is complete. Now we obtain the \emph{Riemannian Schwarzschild--Tangherlini manifold}
$$N_0^n=L_0^2\times_\varphi\Sp^{n-2},\qquad g_{N_0}=dt^2+\varphi'(t)^2\,d\theta^2+\varphi^2g_{\Sp^{n-2}}.$$
It is complete since the base $L_0^2$ is complete, the fibre $\Sp^{n-2}$ is compact, and $\varphi\geq\mu>0$. Meanwhile, $g_{N_0}$ extends smoothly across $B:=\{o\}\times\Sp^{n-2}\cong\Sp^{n-2}(\mu)$, which is totally geodesic since $\nabla\varphi(o)=0$. By \eqref{eq:varphi-RF}, restriction of $N_0^n$ to $t>0$, namely $N_0^n\setminus B$, belongs to the local Ricci-flat family of Subsection~\ref{sec:EWP-def}, so $\oRic^{N_0}=0$ everywhere by continuity.

On $N_0^n\setminus B$, let $r=\varphi(t)$ and $p(r)=1-\left(\frac{\mu}{r}\right)^{n-3}$, then
\begin{equation}\label{eq:gN0}
	g_{N_0}=p(r)\,d\theta^2+p(r)^{-1}dr^2+r^2g_{\Sp^{n-2}},\qquad r>\mu.
\end{equation}

\begin{remark}\label{rem:period}
	The metric \eqref{eq:gN0} is obtained by the Wick rotation $\tau=\oi\theta$ of the Lorentzian Schwarzschild--Tangherlini metric \cite{Schwarzschild16,Tangherlini63}
	\begin{equation*}
		g_{\mathrm{Lor}}=-p(r)\,d\tau^{2}+p(r)^{-1}dr^{2}+r^{2}g_{\Sp^{n-2}},\qquad \tau\in\R, \quad r>\mu,
	\end{equation*}
	on the static exterior region of the $n$-dimensional black hole. Smoothness at $r=\mu$ requires the Euclidean time $\theta$ to have period $\frac{4\pi\mu}{n-3}$, the inverse Hawking temperature \cite{Hawking77}. In physical units, the scale parameter $\mu$ is related to the black-hole mass $\mathcal M$ by $\mu^{n-3}=\frac{16\pi G\,\mathcal{M}}{(n-2)\,\Omega_{n-2}}$, where $G$ is the gravitational constant and $\Omega_{n-2}=\frac{2\pi^{(n-1)/2}}{\Gamma((n-1)/2)}$ is the area of $\Sp^{n-2}$. Our normalization $\mu=\frac{n-3}{2}$ gives Euclidean time period $2\pi$, and the added locus is $B$.
\end{remark}

Applying Proposition~\ref{pro:N} on $N_0^n\setminus B$ gives
\begin{equation}\label{eq:RN0}
	|R^{N_0}|^2=(n-1)(n-2)^2(n-3)\frac{\mu^{2n-6}}{r^{2n-2}}.
\end{equation}
The right-hand side is strictly decreasing for $r>\mu$ and is positive as $r\downarrow\mu$. Since $R^{N_0}$ extends smoothly across $B$, it follows that $R^{N_0}$ is nonzero at every point of $N_0^n$, including the points of $B$. Every nonempty open subset of $N_0$ contains points with distinct values of $r$. Hence, $(N_0^n, g_{N_0})$ is nowhere locally curvature homogeneous and, in particular, curvature inhomogeneous.

We next recall its rotational isometric immersion from~\cite[Example~2]{DOV}.  By \eqref{eq:varphi-RF} and \eqref{eq:varphi-RF2}, the function $Q:=1-\varphi'^2-\varphi''^2$ is nonnegative and is positive for $t>0$. Indeed, one has $Q=(\frac{\mu}{\varphi})^{n-3}-(\frac{\mu}{\varphi})^{2n-4}$. Using \eqref{eq:varphi-expansion}, we obtain
$$Q(t)=\frac{n-1}{n-3}\,t^2+O(t^4)\qquad\text{as }t\to0,$$
and, in particular, $Q(0)=Q'(0)=0$ and $Q''(0)=\frac{2n-2}{n-3}>0$. Since $Q$ is even, its positive square root on $t>0$ has a smooth odd extension $\eta$ across $0$. Define
$$\psi(t)=\int_0^t\eta(s)\,ds.$$
Then $\psi$ is smooth and even, $\psi'>0$ on $(0,\infty)$, and $\psi'^2=1-\varphi'^2-\varphi''^2$ with $\psi(0)=0$.
In polar coordinates on $L_0^2\setminus\{o\}$, define
\begin{equation}\label{eq:fN0-global}
	f_{N_0}(t,\theta,y)=\bigl(\psi(t),\,\varphi'(t)\cos\theta,\,\varphi'(t)\sin\theta,\,\varphi(t)y\bigr)\in\R^{n+2}.
\end{equation}
By Whitney's theorem \cite{Whitney43}, every smooth even function of $t$ is a smooth function of $t^2$ near the origin. Hence, the even functions $\psi$ and $\varphi$ define smooth radial functions on $L_0^2$. Moreover, since $\varphi'$ is smooth and odd, $\frac{\varphi'(t)}{t}$ is likewise a smooth function of $t^2$ with $\varphi''(0)=1$. In Cartesian coordinates $x^1=t\cos\theta$ and $x^2=t\sin\theta$, the middle block in \eqref{eq:fN0-global} is $\frac{\varphi'(t)}{t}(x^1,x^2)$ and is therefore smooth at $o$. Thus, \eqref{eq:fN0-global} extends smoothly across $B$. At $t=0$, its middle block differentiates as the identity on the two Cartesian directions, while the last block has rank $n-2$ on $T\Sp^{n-2}$, and hence the extension has rank $n$. A direct computation on $N_0^n\setminus B$ gives $f_{N_0}^{*}\langle\cdot,\cdot\rangle=g_{N_0}$. Since both sides are smooth on $N_0^n$, this identity extends across $B$ by continuity. Thus, $f_{N_0}$ is a global isometric immersion. Its normal bundle is flat on $N_0^n$ by \cite[Example~2]{DOV}.

\begin{proposition}\label{pro:embedding}
	The map $f_{N_0}\colon N_0^n\to\R^{n+2}$ of \eqref{eq:fN0-global} is an isometric embedding onto a submanifold of $\R^{n+2}$, and its codimension two is the smallest.
\end{proposition}

\begin{proof}
	We only need to verify injectivity and properness. Throughout we use that $\varphi\colon[0,\infty)\to[\mu,\infty)$ is smooth with $\varphi'>0$ on $(0,\infty)$, hence strictly increasing, and that $\varphi(t)\to\infty$ as $t\to\infty$. Thus, $\varphi$ is a homeomorphism onto $[\mu,\infty)$.
	
	Since $\varphi(t)>0$ is strictly increasing, the last block determines $t$ from its norm $\varphi(t)$ and then determines $y$ by division by $\varphi(t)$. If $t>0$, the middle block determines $\theta$ because $\varphi'(t)>0$, while for $t=0$ the angular variable collapses at $o$. Thus, $f_{N_0}$ is injective.
	
	The function $t\colon N_0^n\to[0,\infty)$ is proper. Indeed,  $g_{L_0}$ is written in geodesic polar form, so $t$ is the distance to $o$ in the complete surface $(L_0^2,g_{L_0})$. By the Hopf--Rinow theorem, the closed ball $\overline{\B}^2_a(o)\subset L_0^2$ is compact, and therefore so is
	$$\{t\leq a\}=\overline{\B}^2_a(o)\times\Sp^{n-2}\subset N_0^n$$
	for every $a\geq0$. The computation above gives $|f_{N_0}(x)|\geq\varphi(t(x))$ for all $x$. Let $K\subset\R^{n+2}$ be compact, and choose $b\geq\mu$ with $K\subset\overline{\B}^{n+2}_b(0)\subset\R^{n+2}$. Then
	$$f_{N_0}^{-1}(K)\subseteq\{x\in N_0^n:\varphi(t(x))\leq b\}=\{t\leq\varphi^{-1}(b)\}.$$
	Thus, $f_{N_0}^{-1}(K)$ is a closed subset of a compact set, hence compact, and $f_{N_0}$ is proper.

	An injective proper immersion is a closed embedding, so $f_{N_0}$ is a global isometric embedding of codimension two onto a submanifold (closed as a subset) of $\R^{n+2}$. On every connected open subset, codimension zero is impossible because $N_0^n$ is nonflat, while the curvature inhomogeneity established after \eqref{eq:RN0} and Corollary~\ref{cor:nohyp} rules out codimension one. Hence, the codimension two is the smallest even locally.
\end{proof}

\begin{remark}\label{rem:profile}
	For the coordinate expression \eqref{eq:gN0} on $N_0^n\setminus B$, set $s(r)=\sqrt{p(r)}$ and
	\begin{equation*}
		Z(r)=\int_\mu^r\sqrt{p(u)^{-1}-1-s'(u)^2}\,du.
	\end{equation*}
	One can rewrite the codimension-two isometric embedding \eqref{eq:fN0-global} as
	\begin{equation}\label{eq:fN0}
		\Psi_{N_0}(r,\theta,y)=\bigl(s(r)\cos\theta,\,s(r)\sin\theta,\,Z(r),\,r y\bigr),\qquad r>\mu.
	\end{equation}
	Here, the first three ambient coordinates have been cyclically permuted to facilitate comparison with the physics literature, and this does not affect the geometric properties of the embedding. For $n=4$, the Lorentzian antecedents of~\eqref{eq:fN0} go back to Kasner~\cite{Kasner21} and Fronsdal~\cite{Fronsdal59}, while the full-symmetry-equivariant codimension-two Lorentzian embeddings were later classified by Paston and Sheykin~\cite{PastonSheykin2012}. Under $\mu=\frac{1}{2}$, \eqref{eq:fN0} agrees with the Riemannian Wick rotation of Fronsdal's embedding. The detailed comparison is given in Appendix~\ref{app:comparison}.
\end{remark}

\section{Proof of the main results}\label{sec:proofs}
In this section, we prove Theorems~\ref{thm:1},~\ref{thm:2} and~\ref{thm:3}.

\subsection{Cylindrical products of codimension two: proof of Theorem~\ref{thm:1}}\label{sec:cyl}
\begin{proof}
Let $(N^n,g_N)$ be the non-Ricci-flat Einstein warped product of Subsection~\ref{sec:EWP-def} ($\rho\neq 0$ and $c\neq 0$), and let
$M^{n+\ell}=N^n\times\R^\ell$. The curvature tensor of $M$ satisfies
$$R^M\bigl((X,V),(Y,W)\bigr)(Z,U)=\bigl(R^N(X,Y)Z,\,0\bigr),$$
for $(X,V),(Y,W),(Z,U)\in T_xN^n\oplus T_z\R^\ell$. Therefore, the Ricci curvature is
\begin{equation*}
	\oRic^M=\oRic^N\oplus 0=\rho\,g_N\oplus 0.
\end{equation*}
Since $\rho\neq 0$, $M$ has exactly two distinct constant Ricci eigenvalues, namely $\rho$ and $0$, with multiplicities $n$ and $\ell$, respectively. This proves part (i).

The Kretschmann scalar of $M$ is
$$|R^M|^2(x,z)=|R^N|^2(x),$$
which is nonconstant on every nonempty open set by Proposition \ref{pro:N}. Thus $M$ is nowhere locally curvature homogeneous, proving part (ii).

For part (iii), consider the cylindrical map
$$F=f\times\operatorname{id}_{\R^\ell}\colon M^{n+\ell}\longrightarrow\R^{n+2}\times\R^\ell=\R^{n+\ell+2},\qquad F(x,z)=(f(x),z),$$
where $f\colon N^n\to\R^{n+2}$ is the codimension-two isometric immersion of Subsection~\ref{sec:EWP-def}. Since $dF_{(x,z)}(X,V)=(df_xX,V)$ and $f$ is an isometric immersion, so is $F$, with normal spaces $\nu_{(x,z)}F=\nu_xf\oplus\{0\}$. In particular, the normal bundle of $F$ is flat. Codimension zero is impossible because $M$ is nonflat. By part (ii), every nonempty open subset of $M$ is curvature inhomogeneous, while its Ricci eigenvalues remain constant. Consequently, Corollary~\ref{cor:nohyp}, applied to any sufficiently small connected open subset, rules out a local codimension-one immersion, so the codimension two attained by $F$ is the minimum local Euclidean codimension. This proves part (iii) and completes the proof of Theorem~\ref{thm:1}.
\end{proof}

\subsection{Spherical products of codimension \texorpdfstring{$\geq 2$}{>=2}: proofs of Theorems~\ref{thm:2} and \ref{thm:3}}\label{sec:sph}
\begin{proof}[Proof of Theorem~\ref{thm:2}]
Let $(N_0^n,g_{N_0})$ be the complete manifold of Subsection~\ref{sec:RFS}, and let $W_k^{m}=N_0^n\times\prod_{j=1}^k\Sp^{n_j}(r_j)$, where $k\geq 1$, $n_j\geq 2$, and the radii $r_j>0$ are chosen so that the constants $\rho_j=\frac{n_j-1}{r_j^2}$ are pairwise distinct. The Riemannian product $W_k^{m}$ is complete. Its Ricci tensor is
\begin{equation*}
	\oRic^{W_k}=0\oplus\bigoplus_{j=1}^k\rho_j\,g_{\Sp^{n_j}(r_j)}.
\end{equation*}
Hence, $W_k^{m}$ has exactly $k+1$ distinct constant Ricci eigenvalues, namely $0,\rho_1,\ldots,\rho_k$, with multiplicities $n,n_1,\ldots,n_k$. This proves part~(i).

The Kretschmann scalar is
$$|R^{W_k}|^2=|R^{N_0}|^2+\sum_{j=1}^k\frac{2n_j(n_j-1)}{r_j^4},$$
and the first term is nonconstant on every nonempty open set by \eqref{eq:RN0}. Thus, $W_k^m$ is nowhere locally curvature homogeneous, proving part~(ii).

Let $\iota_j\colon\Sp^{n_j}(r_j)\hookrightarrow\R^{n_j+1}$ be the standard embedding. The extrinsic product
\begin{equation}\label{eq:Fk}
	F_k=f_{N_0}\times\iota_1\times\cdots\times\iota_k\colon W_k^{m}\longrightarrow\R^{n+2}\times\R^{n_1+1}\times\cdots\times\R^{n_k+1}=\R^{m+k+2}
\end{equation}
is a global isometric embedding, because each factor map is an embedding (see Proposition~\ref{pro:embedding}). Its normal bundle is the direct sum of the flat normal bundle of $f_{N_0}$ and the $k$ normal line bundles of the round-sphere embeddings. Hence, it is flat and has rank $k+2$. Its second fundamental form is adapted to the product structure, and Theorem~\ref{thm:3}(a) below shows that $k+2$ is the smallest possible codimension within that class. This proves part~(iii) and completes the proof of Theorem~\ref{thm:2}.
\end{proof}

It remains to establish the sharpness of the codimension. We keep the notation $P_0=N_0^n$, $P_j=\Sp^{n_j}(r_j)$ of the introduction, so that $TW_k^m=\bigoplus_{j=0}^k TP_j$.

\begin{proof}[Proof of Theorem~\ref{thm:3}]
	(a) Let $\Phi$ be an isometric immersion of codimension $q$ whose second fundamental form $\alpha:=\alpha_\Phi$ is adapted, i.e., $\alpha(TP_i,TP_j)=0$ for $i\neq j$. Fix a point of $W_k^m$, and all spaces in the following argument are understood at this point. For unit vectors $X\in TP_i$ and $Y\in TP_j$ with $i\neq j$, the Gauss equation gives
	$$0=K^{W_k}(X,Y)=\langle\alpha(X,X),\alpha(Y,Y)\rangle-|\alpha(X,Y)|^2=\langle\alpha(X,X),\alpha(Y,Y)\rangle.$$
	Polarizing first in $X$ and then in $Y$ shows that the subspaces
	$$E_i:=\operatorname{span}\{\alpha(X_1,X_2):X_1,X_2\in TP_i\},\qquad 0\leq i\leq k,$$
	are mutually orthogonal. Each sphere factor $P_j$, $1\leq j\leq k$, has positive sectional curvature $\frac{1}{r_j^2}$, so $E_j\neq0$. These $k$ subspaces occupy at least $k$ normal dimensions, leaving $\dim E_0\leq q-k$.
	
	Assume, to the contrary, that $q\leq k+1$. Then $\dim E_0\leq1$. If $E_0=0$, the Gauss equation immediately gives $R^{N_0}=0$. Otherwise, let a unit normal vector $\xi$ span $E_0$ and write $A:=A_\xi|_{TN_0}$ for the restriction to $TN_0^n$ of the shape operator of $\xi$. Then $\alpha(X,Y)=\langle AX,Y\rangle\,\xi$ for all $X,Y\in TN_0^n$. On $TN_0^n$, the Gauss equation takes the hypersurface form
	\begin{equation}\label{eq:R}
		R^{N_0}(X,Y,Z,V)=\langle AX,V\rangle\langle AY,Z\rangle-\langle AX,Z\rangle\langle AY,V\rangle.
	\end{equation}
	Since $A$ is self-adjoint and $N_0^n$ is Ricci-flat, contraction of \eqref{eq:R} yields $0=(\tr A)A-A^2$. If $\lambda_1,\ldots,\lambda_n$ are the eigenvalues of $A$, then $\lambda_i(\tr A-\lambda_i)=0$ for every $i$. Every nonzero eigenvalue equals $\tr A$, which is impossible for two or more nonzero eigenvalues. Hence, at most one eigenvalue of $A$ is nonzero. Equation~\eqref{eq:R} again forces $R^{N_0}=0$, contradicting the nonvanishing established after \eqref{eq:RN0}. Therefore $q\geq k+2$. The extrinsic product $F_k$ in \eqref{eq:Fk} attains this bound, without any assumption that the $\rho_j$ are distinct.

	(b) Assume the $\rho_j$ are pairwise distinct, and $\Phi\colon W_k^{m}\to\R^{m+q}$ is a local isometric immersion with $q\leq k+1$. Part (a) shows that its second fundamental form is non-adapted. Let $\{\xi_1,\ldots,\xi_q\}$ be a local orthonormal normal frame with corresponding shape operators $A_{\xi_a}$, $1\leq a\leq q$. If the normal bundle were flat, then the Ricci equation gives $[A_{\xi_a},A_{\xi_b}]=0$ for all $a,b$. Since $\Phi$ is a Euclidean isometric immersion, the Gauss equation yields 
	$$\widehat{\oRic}=\sum_{a=1}^q\bigl((\tr A_{\xi_a})A_{\xi_a}-A_{\xi_a}^2\bigr),$$
	where $\widehat{\oRic}$ is the Ricci operator of $W_k^{m}$. Since the shape operators commute pairwise, every $A_{\xi_a}$ commutes with $\widehat{\oRic}$. By Theorem~\ref{thm:2}(i), $\widehat{\oRic}$ has the pairwise distinct eigenvalues $0,\rho_1,\ldots,\rho_k$, with eigenspaces exactly $TN_0^n,T\Sp^{n_1}(r_1),\ldots,T\Sp^{n_k}(r_k)$. Hence, each $A_{\xi_a}$ preserves every factor $TP_i$. Therefore, for $X\in TP_i$ and $U\in TP_j$ with $i\neq j$,
	$$\langle\alpha(X,U),\xi_a\rangle=\langle A_{\xi_a}X,U\rangle=0,\qquad 1\leq a\leq q.$$
	Since $\xi_1,\ldots,\xi_q$ span the normal space, $\alpha$ is adapted, contradicting the first assertion. Thus, the normal bundle of $\Phi$ is not flat.
\end{proof}

\section{Further questions}\label{sec:questions}
The results suggest three natural questions. First, does $W_k^m$ admit a local isometric immersion into Euclidean space of codimension lower than $k+2$? By Theorem~\ref{thm:3}(b), any such immersion would necessarily have non-adapted second fundamental form and nonflat normal bundle. Second, do constant Ricci eigenvalues force curvature homogeneity for submanifolds with at least three distinct Ricci eigenvalues and minimum Euclidean codimension two? The restriction to at least three eigenvalues is natural. For exactly two distinct eigenvalues, Theorem~\ref{thm:1} already provides curvature inhomogeneous examples of minimum Euclidean codimension two. On the other hand, the examples in Theorem~\ref{thm:2} have codimension $k+2\geq3$, leaving the codimension-two case open. Third, complementary to the preceding two questions, can one characterize high-codimensional curvature homogeneous submanifolds, especially isoparametric submanifolds, of real space forms?

\appendix
\section{Comparison with the classical four-dimensional Schwarzschild embeddings}\label{app:comparison}
This appendix verifies the compatibility of our Riemannian embedding \eqref{eq:fN0} with the Lorentzian Schwarzschild embeddings of Kasner \cite{Kasner21} and Fronsdal \cite{Fronsdal59} types in the uniform parametrization of Paston--Sheykin \cite{PastonSheykin2012}. Let $n=4$, keep the notation of Subsection \ref{sec:RFS} with an arbitrary scale $\mu>0$, then
$$p(r)=1-\frac{\mu}{r},\qquad s(r)=\sqrt{p(r)},\qquad r>\mu.$$
The Lorentzian and Riemannian metrics to be compared are
\begin{align*}
	g_{\rm Lor}&=-p(r)\,d\tau^2+p(r)^{-1}\,dr^2+r^2g_{\Sp^2},\\
	g_{\rm Riem}&=\phantom{-}p(r)\,d\theta^2+p(r)^{-1}\,dr^2+r^2g_{\Sp^2}.
\end{align*}

Denote $\R^{p,q}$ by the flat space with $p$ timelike and $q$ spacelike directions. For constants $\alpha,\beta>0$, the exterior parts of the Kasner and Fronsdal types may be written respectively as
\begin{align}
	\Psi_K(r,\tau,y)
	&=\left(\alpha^{-1}s(r)\sin(\alpha\tau),
	\alpha^{-1}s(r)\cos(\alpha\tau),Z_K(r),r y\right)\in\R^{2,4},
	\label{eq:app-Kasner}\\
	\Psi_F(r,\tau,y)
	&=\left(\beta^{-1}s(r)\sinh(\beta\tau),
	\beta^{-1}s(r)\cosh(\beta\tau),Z_F(r),r y\right)\in\R^{1,5},
	\label{eq:app-Fronsdal}
\end{align}
where $y\in\Sp^2$ and
\begin{equation}\label{eq:app-heights}
	Z_K'(r)^2=p(r)^{-1}-1+\alpha^{-2}s'(r)^2,\qquad
	Z_F'(r)^2=p(r)^{-1}-1-\beta^{-2}s'(r)^2.
\end{equation}
Here $\alpha$ is free in the Kasner exterior embedding, and the original paper \cite{Kasner21} takes $\alpha=1$. Smoothness of the Fronsdal embedding across the Lorentzian horizon, namely $r=\mu$, selects $\beta=\frac{1}{2\mu}$ \cite{Fronsdal59}, and then
\begin{equation*}
	Z_F'(r)^2=p(r)^{-1}-1-4\mu^2s'(r)^2
	=\frac{\mu(r^2+\mu r+\mu^2)}{r^3}.
\end{equation*}

At the same scale, the Riemannian embedding relevant to \eqref{eq:fN0} of this paper is
\begin{equation}\label{eq:app-Riemannian-general}
	\Psi_{N_0,\mu}(r,\theta,y)=
	\left(2\mu s(r)\cos\frac{\theta}{2\mu},
	2\mu s(r)\sin\frac{\theta}{2\mu},
	Z_F(r),r y\right).
\end{equation}
It induces $g_{\rm Riem}$, and it extends smoothly across the bolt $B$ when $\theta$ has period $4\pi\mu$ by Remark \ref{rem:period}. Formula \eqref{eq:app-Riemannian-general} is precisely the Riemannian Wick rotation of the horizon-regular Fronsdal embedding \eqref{eq:app-Fronsdal}, since one sets $\tau=\oi\theta$ and converts the timelike ambient coordinate into a spacelike one. By contrast, the Kasner type \eqref{eq:app-Kasner} uses a negative-definite rotation plane, which accounts for the opposite sign of the $s'(r)^2$ term in \eqref{eq:app-heights}.

In the body of the paper, we normalize the Euclidean angular period to be $2\pi$. Since smoothness requires $4\pi\mu=2\pi$, this gives $\mu=\frac{1}{2}$, and \eqref{eq:app-Riemannian-general} reduces, up to an ambient permutation and rotation, to exactly \eqref{eq:fN0}. Thus $\mu=\frac{1}{2}$ reflects only the normalization adopted in the present paper. The comparison with the Paston--Sheykin formulas is valid at every common scale $\mu>0$.

\end{document}